\documentclass[12pt,a4paper]{amsart}
\usepackage{graphicx,amssymb}
\usepackage{amsmath,amssymb,amscd}
\input xy
\xyoption{all}
\usepackage[all]{xy}
\usepackage{tikz}

\usepackage{hyperref}

\newcommand{\lra}{\longrightarrow}

\newcommand{\ra}{\rightarrow}

\newcommand{\Hom}{\operatorname{Hom}}

\newcommand{\Cliff}{\operatorname{Cliff}}
\newcommand{\gr}{\operatorname{gr}}

\newcommand{\Coker}{\operatorname{Coker}}
\newcommand{\Ker}{\operatorname{Ker}}
\theoremstyle{plain}
\newtheorem{theorem}{Theorem}[section]
\newtheorem{lem}[theorem]{Lemma}
\newtheorem{prop}[theorem]{Proposition}
\newtheorem{cor}[theorem]{Corollary}

\newtheorem{rem}[theorem]{Remark}
\newtheorem{ex}[theorem]{Example}
\numberwithin{equation}{section}
\begin{document}
\title[Rank-2 Brill-Noether theory]{Some examples of rank-2 Brill-Noether loci}

\author{P. E. Newstead}

\address{P.E. Newstead\\Department of Mathematical Sciences\\
              University of Liverpool\\
              Peach Street, Liverpool L69 7ZL, UK}
\email{newstead@liv.ac.uk}

\date{\today}
\thanks{}

\begin{abstract}
In this paper, we construct some examples of rank-2 Brill-Noether loci with ``unexpected'' properties on general curves. The key examples are in genus 6, but we also have interesting examples in genus 5 and in higher genus. We relate some of our results to the recent proof of Mercat's conjecture in rank 2 by Bakker and Farkas. 
\end{abstract}
\maketitle

\section{Introduction}\label{sec-intro}
Let $C$ be a general curve of genus $g$ defined over the complex numbers. The main focus of this paper is to study certain rank-2 Brill-Noether loci in the case $g=6$ and, in particular, to show that $B(2,10,4)$ is reducible (see below for the definitions); this is contrary to na\"ive expectations. We consider also similar situations in genus $5$ and in higher genus and finish with some results on bundles computing the rank-2 Clifford index for low values of $g$. These examples are presented as a contribution to higher rank Brill-Noether theory, which is still far from fully understood even in rank $2$.

We denote by $M(n,d)$ (respectively, $\widetilde{M}(n,d)$) the moduli space of stable bundles (respectively, S-equivalence classes of semistable bundles) of rank $n$ and degree $d$ on $C$ and define
\begin{gather*}
B(n,d,k):=\{E\in M(n,d)|h^0(E)\ge k\}\\
\widetilde{B}(n,d,k):=\{[E]\in \widetilde{M}(n,d)|h^0(\gr(E))\ge k\}.
\end{gather*}
(Here $[E]$ denotes the S-equivalence class of a semistable bundle and $\gr(E)$ denotes the graded object associated with $E$.) We write also $K_C$ for the canonical bundle of $C$ and $B(2,K_C,k)$ ($\widetilde{B}(2,K_C,k)$) for the subvariety of $B(2, 2g-2,k)$ ($\widetilde{B}(2,2g-2,k)$) given by bundles of determinant $K_C$. Our first main result is

\ 

\noindent{\bf Theorem \ref{t-g6k4}.}
\emph{Let $C$ be a general curve of genus $g=\lambda(2\lambda-1)$ for $\lambda\in\mathbb{Z}$, $\lambda\ge2$. Then $B(2,K_C,2\lambda)$ has pure dimension $4\lambda(\lambda-1)-3$ and is smooth outside the non-empty locus $B(2,K_C,2\lambda+1)$. Moreover $B(2,2g-2,2\lambda)$ has at least one irreducible component of dimension $4\lambda(\lambda-1)-3$ which is not contained in $B(2,K_C,2\lambda)$.}

\

This is of particular significance in the case $\lambda=2$ or equivalently $g=6$, which is the first value of the genus for which the expected dimension of $B(2,2g-2,k)$ can be negative while that of $B(2,K_C,k)$ is non-negative. The appropriate value of $k$ in this case is $k=5$ and we prove

\

\noindent{\bf Theorem \ref{p-g6k5}.}
\emph{Let $C$ be a general curve of genus $6$. Then
$\widetilde{B}(2,10,k)=\emptyset$ for $k\ge6$. Moreover
$B(2,10,5)=\widetilde{B}(2,10,5)=B(2,K_C,5)$ consists of a single point $E_{2,10,5}$ and $E_{2,10,5}$ is generated.}

\

We show further that $B(3,10,5)$ consists of a single point (Proposition \ref{p-r3}). Also in Section \ref{sec-g6}, we relate our results for genus $6$ to others in the literature and interpret them in terms of coherent systems. 

In Section \ref{sec-g5}, we consider a somewhat analogous problem for $g=5$. Finally, in Section \ref{sec-cliff}, we obtain some results on bundles computing rank-2 Clifford indices for low values of $g$ which extend those of \cite{ln3} and relate them to the recent result of Bakker and Farkas \cite{bfa} confirming Mercat's conjecture in rank $2$ for general curves.

My thanks are due to the referee(s) for some useful suggestions.

\section{Background and preliminaries}\label{sec-pre}
Throughout the paper, $C$ will be a smooth curve of genus $g\ge5$ defined over the complex numbers. For any vector bundle $E$ on $C$, we write $n_E$ for the rank of $E$ and $d_E$ for the degree of $E$. We define 
\[\Cliff(E):=\frac1{n_E}\left(d_E-2(h^0(E)-n_E)\right)\]
and
\begin{eqnarray*}
\Cliff_n(C):&=&\min\big\{\Cliff(E)|E \mbox{ semistable},\\ 
&& n_E=n, h^0(E)\ge2n_E, d_E\le n_E(g-1)\big\}.
\end{eqnarray*}
With this notation, $\Cliff_1(C)$ is the classical Clifford index $\Cliff(C)$. We recall that, for $C$ a general curve of genus $g$, $\Cliff(C)=\lfloor\frac{g-1}2\rfloor$ and the gonality of $C$ (the minimal degree of a line bundle with $h^0\ge2$) is $\operatorname{gon}(C)=\lfloor\frac{g-1}2\rfloor+2$. It is clear that $\Cliff_n(C)\le \Cliff(C)$ for all $n$, and Mercat \cite{m} conjectured that $\Cliff_n(C)=\Cliff(C)$ (actually Mercat's conjecture is a little stronger than this (see \cite[Proposition 3.3]{ln}), but equivalent to it in rank $2$). There are many counter-examples to this conjecture, but recently Bakker and Farkas \cite{bfa} have proved that, for $C$ a general curve of genus $g$,
\begin{equation}\label{eqc2}
\Cliff_2(C)=\Cliff(C).
\end{equation}

The Brill-Noether locus $B(n,d,k)$ has an ``expected'' dimension
\[\beta(n,d,k):=n^2(g-1)+1-k(k-d+n(g-1)).\]
Provided $d<n(g-1)+k$, every irreducible component of $B(n,d,k)$ has dimension $\ge\beta(n,d,k)$. The infinitesimal behaviour of $B(n,d,k)$ is governed in part by the multiplication map (often referred to as the Petri map)
\[H^0(E)\otimes H^0(E^*\otimes K_C)\lra H^0(E\otimes E^*\otimes K_C).\]
In fact, $B(n,d,k)$ is smooth of dimension $\beta(n,d,k)$ at a point $E$ if and only if the Petri map is injective. For $n=1$, one can define a {\em Petri curve} to be a curve for which
\[H^0(L)\otimes H^0(L^*\otimes K_C)\lra H^0(K_C)\]
is injective for all line bundles $L$. The general curve of any genus is a Petri curve and, if $C$ is Petri and $d<g-1+k$, $B(1,d,k)$ is empty if $\beta(1,d,k)<0$, of dimension $\beta(1,d,k)$ if $\beta(1,d,k)\ge0$ and irreducible if $\beta(1,d,k)>0$. Moreover, if $\beta(1,d,k)\ge0$, the singular set of $B(1,d,k)$ is $B(1,d,k+1)$. (For these and other results in classical Brill-Noether theory, see \cite{acgh}.) There is no analogue of these results for higher rank.

For $B(2,K_C,k)$, the expected dimension is not $\beta(2,2g-2,k)-g$, but instead it is
\[\beta(2,K_C,k)=3g-3-\frac{k(k+1)}2.\]
There is also a different Petri map (obtained by symmetrizing the usual Petri map with respect to the natural isomorphism $E\simeq E^*\otimes K_C$)
\[\operatorname{Sym}^2(H^0(E))\lra H^0(\operatorname{Sym}^2(E)).\]
One can then prove that, on a general curve, this Petri map is always injective for stable $E$ and hence $B(2,K_C,k)$ is smooth at any point $E$ for which $h^0(E)=k$ (see \cite{te2}). There are also partial results on non-emptiness for $B(2,K_C,k)$ for all $g$ \cite{te1} (see also \cite{lnp,zh}) and complete results for small values of $g$ \cite{bf}. Some detailed results for $k\le3$ can be found in \cite[section 7]{cf} and for $k=4$ in \cite{cra}.

By a {\em subpencil} of a bundle $E$, we mean a rank-$1$ subsheaf $L$ of $E$ such that $h^0(L)=2$. The following lemmas will be useful.
\begin{lem}\label{l-p1}
Let $E$ be a bundle of rank $2$ on $C$ such that $h^0(E)=s+2$, $s\ge1$. If $E$ does not admit a subpencil, then $h^0(\det E)\ge 2s+1$.
\end{lem}
\begin{proof}
This is the rank-2 case of \cite[Lemma 3.9]{pr}.
\end{proof}

\begin{cor}\label{c-p1}
Let $C$ be a general curve of genus $g\ge6$ and $E$ a semistable bundle with $d_E=2g-2$ which computes $\Cliff_2(C)$. If $g=9$, suppose in addition that $E$ is stable. If either $g$ is even or $\det E\not\simeq K_C$, then $E$ is expressible in the form
\begin{equation}\label{eq-p1}
0\lra L\lra E\lra L'^*\otimes K_C\lra0,
\end{equation}
where $d_L=d_{L'}=\left\lfloor\frac{g-1}2\right\rfloor+2$ and $h^0(L)=h^0(L')=2$. Moreover, all sections of $L'^*\otimes K_C$ lift to $E$.
\end{cor}
\begin{proof} Suppose first that $g=2s$, so that $\Cliff_2(C)=\Cliff(C)=s-1$ and $h^0(E)=s+2$. If $E$ does not admit a subpencil, then, by Lemma \ref{l-p1}, $h^0(\det E)\ge 2s+1=g+1$, a contradiction. Now suppose that $g=2s+1$, so that $\Cliff_2(C)=s$ and again $h^0(E)=s+2$. Now, by Lemma \ref{l-p1}, $h^0(\det E)\ge 2s+1=g$. Since $\det E\not\simeq K_C$, this is again a contradiction. So $E$ admits a subpencil.

For $g=6$, the only possibility is given by \eqref{eq-p1}. For $g\ge7$, the existence of \eqref{eq-p1} follows from \cite[Proposition 7.2 and Theorem 7.4]{ln3}. Using Riemann-Roch, it is easy to check that 
\[h^0(L)+h^0(L'^*\otimes K_C)=s+2=h^0(E),\]
so all sections of $L'^*\otimes K_C$ must lift to $E$.
\end{proof}

\begin{lem}\label{l-p2}
Let $C$ be a Petri curve of genus $g$ and $L$ a line bundle with $d_L=\operatorname{gon}(C)$ and $h^0(L)=2$. Then
\begin{enumerate}
\item $L$ is generated, in other words, the evaluation map 
\[H^0(L)\otimes\mathcal{O}_C\ra L\] 
is surjective;
\item if $g$ is even, $h^0(L\otimes L)=3$;
\item if $g$ is odd, $h^0(L\otimes L)=4$.
\end{enumerate}
\end{lem}
\begin{proof} (1) is obvious, since otherwise $h^0(L(-p))=2$ for some $p\in C$, contradicting the definition of $\operatorname{gon}(C)$. For (2) and (3), see \cite[Lemma 2.10]{ln3}.
\end{proof}

\begin{lem}\label{l-p3}
There exist non-split exact sequences
\[0\lra E_1\lra E\lra E_2\lra0\]
of vector bundles for which all sections of $E_2$ lift to sections of $E$ if and only if the multiplication map
\[m:H^0(E_2)\otimes H^0(E_1^*\otimes K_C)\lra H^0(E_2\otimes E_1^*\otimes K_C)\]
fails to be surjective. Such extensions are classified up to isomorphism by ${\mathbb P}((\Coker m)^*)$. 
\end{lem}
\begin{proof}
The extensions for which all sections lift are classified by the kernel of the natural map
\[H^1(E_2^*\otimes E_1)\lra\Hom(H^0(E_2),H^1(E_1)).\]
The map $m$ is the dual of this map.
\end{proof}

The following lemma is undoubtedly well known, but I have been unable to locate a reference.

\begin{lem}\label{l-p4}
Let $F$ be a vector bundle on $C$ with $h^1(F)\ge r$ for some positive integer $r$. Then, for $\tau$ a general torsion sheaf of length $t\le r$ and 
\begin{equation}\label{eq-tor}
0\lra F\lra E\lra\tau\lra 0
\end{equation}
a general extension of $\tau$ by $F$, $H^0(E)=H^0(F)$. 
\end{lem}
\begin{proof}
By induction, it s clearly sufficient to prove this when $t=1$. In this case $\tau=\mathbb{C}_p$ for a general point $p\in C$. Dualising \eqref{eq-tor} and tensoring by $K_C$, we obtain an exact sequence
\[0\lra E^*\otimes K_C\lra F^*\otimes K_C\lra \mathbb{C}_p\lra0.\]
Now note that $h^0(F^*\otimes K_C)\ne0$ and for general $p$ and the general homomorphism 
$F^*\otimes K_C\to\mathbb{C}_p$, the map $H^0(F^*\otimes K_C)\to\mathbb{C}_p$ is non-zero. It follows that $h^0(E^*\otimes K_C)=h^0(F^*\otimes K_C)-1$ and so $h^0(E)=h^0(F)$, giving the required result.
\end{proof}

Finally, we recall that a {\em coherent system} on $C$ of type $(n,d,k)$ is a pair $(E,V)$ consisting of a vector bundle $E$ of rank $n$ and degree $d$ and a subspace $V$ of $H^0(E)$ of dimension $k$. There is a concept of $\alpha$-stability for coherent systems for $\alpha\in \mathbb{R}$ and moduli spaces $G(\alpha;n,d,k)$ and $\widetilde{G}(\alpha;n,d,k)$ exist. (For basic information on this construction, see \cite{bgmn}.) The definition of $\alpha$-stability depends on the $\alpha$-slope of $(E,V)$ defined by $\mu_\alpha(E,V):=\frac{d+\alpha k}n$.

\section{A reducible Brill-Noether locus}\label{sec-hg}

In this section, we prove our first main theorem. While the key case is for curves of genus $6$, the theorem in fact holds for infinitely many values of the genus.

\begin{lem}\label{l-hg1}
Let $C$ be a general curve of genus $g=\lambda(2\lambda-1)$ for $\lambda\in\mathbb{Z}$, $\lambda\ge2$. Then $B(1,g-\lambda,\lambda)$ is a finite set of cardinality
\begin{equation*}
(\lambda(2\lambda-1))!\prod_{i=0}^{\lambda-1}\frac{i!}{(2\lambda+i-1)!}.
\end{equation*} 
Moreover, if $L\in B(1,g-\lambda,\lambda)$, then $L$ is generated and  $h^0(L)=\lambda$.
\end{lem}
\begin{proof}
This follows by classical Brill-Noether theory from the fact that $\beta(1,g-\lambda,\lambda)=0$ (see \cite[p.211 formula (1.2)]{acgh} for the formula for the cardinality).
\end{proof}

For $C$ as in Lemma \ref{l-hg1}, a simple calculation gives
\begin{equation}\label{eq-hg1}
\beta(2,2g-2,2\lambda)=\beta(2,K_C,2\lambda)=4\lambda(\lambda-1)-3.
\end{equation}

\begin{theorem}\label{t-g6k4}
Let $C$ be a general curve of genus $g=\lambda(2\lambda-1)$ for $\lambda\in\mathbb{Z}$, $\lambda\ge2$. Then $B(2,K_C,2\lambda)$ has pure dimension $4\lambda(\lambda-1)-3$ and is smooth outside the non-empty locus $B(2,K_C,2\lambda+1)$. Moreover $B(2,2g-2,2\lambda)$ has at least one irreducible component of dimension $4\lambda(\lambda-1)-3$ which is not contained in $B(2,K_C,2\lambda)$.
\end{theorem}
\begin{proof}
The fact that $B(2,K_C,2\lambda)$ and $B(2,K_C,2\lambda+1)$ are both non-empty is proved in \cite{bf} for $\lambda=2$ and in \cite{ot,te1} for $\lambda\ge3$. The rest of the first assertion is proved in \cite{te2}.

To obtain bundles in $B(2,2g-2,2\lambda)$ which do not have determinant $K_C$, we consider exact sequences
\begin{equation}\label{eq-g62}
0\lra L_1\oplus L_2\lra E\lra \mathbb{C}_{p_1}\oplus\cdots\oplus \mathbb{C}_{p_{2\lambda-2}}\lra0,
\end{equation}
where $L_1$, $L_2$ are distinct elements of $B(1,g-\lambda,\lambda)$ (these exist by Lemma \ref{l-hg1}) and the $p_j$ are distinct points of $C$. The general such extension gives rise to a stable bundle $E$ by \cite[Th\'eor\`eme A-5]{m3}; moreover the homomorphism $E^*\to L_i^*$ obtained by dualising \eqref{eq-g62} is surjective. By Riemann-Roch, $h^1(L_i)=2\lambda-1$, so Lemma \ref{l-p4} implies that, for a general choice of $p_j$ and \eqref{eq-g62},
 \[H^0(E)=H^0(L_1)\oplus H^0(L_2).\] 
Moreover, for general $p_j$, we have 
\[h^0(L_i^*(-p_1-\cdots-p_{2\lambda-2})\otimes K_C)=1.\] 

We now show that, for a generic choice of the $p_j$, the Petri map of $E$ is injective. This will prove that $E$ belongs to a unique irreducible component $B_0$ of dimension $4\lambda(\lambda-1)-3$ (see \eqref{eq-hg1}), which is evidently not contained in $B(2,K_C,4)$. In fact, the Petri map 
\[H^0(E)\otimes H^0(E^*\otimes K_C)\lra H^0(E\otimes E^*\otimes K_C)\]
splits into 
\[\mu_1:H^0(L_1)\otimes H^0(E^*\otimes K_C)\lra H^0(L_1\otimes E^*\otimes K_C)\]
and
\[\mu_2:H^0(L_2)\otimes H^0(E^*\otimes K_C)\lra H^0(L_2\otimes E^*\otimes K_C).\]
It is sufficient to prove that both these maps are injective. Now we have a commutative diagram
\[
\begin{array}{ccc}
H^0(L_1)\otimes H^0(E^*\otimes K_C)&\stackrel{\mu_1}{\lra}&H^0(L_1\otimes E^*\otimes K_C)\\
\Big\downarrow&&\Big\downarrow\\
H^0(L_1)\otimes H^0(L_1^*\otimes K_C)&\lra& H^0(K_C),
\end{array}\]
where the vertical arrows are induced by the homomorphism $E^*\to L_1^*$. The lower horizontal map is injective since $C$ is Petri, so
\[\Ker \mu_1\subset H^0(L_1)\otimes\Ker(H^0(E^*\otimes K_C)\lra H^0(L_1^*\otimes K)).\]
Since $E^*\to L_1^*$ is surjective,
\[\Ker(H^0(E^*\otimes K_C)\lra H^0(L_1^*\otimes K_C))=H^0(L_2^*(-p_1-\cdots-p_{2\lambda-2})\otimes K_C).\]
Moreover,
$h^0(L_2^*(-p_1-\cdots-p_{2\lambda-2})\otimes K_C)=1$, so
\[\mu_1|H^0(L_1)\otimes H^0(L_2^*(-p_1-\cdots p_{2\lambda-2})\otimes K_C)\]
 is injective. Hence $\Ker \mu_1=0$ and $\mu_1$ is injective. The same argument applies to $\mu_2$, completing the proof that $B_0$ has dimension $4\lambda(\lambda-1)-3$.
\end{proof}

\begin{rem}\label{r-g63}
\emph{The fact that $B(2,2g-2,2\lambda)$ has a component of dimension $4\lambda(\lambda-1)-3$ is proved in \cite{te}. The argument in the proof above, using \cite{m3}, is more precise and shows that there is a component not contained in $B(2,K_C,2\lambda)$. On the other hand, it is not proved in \cite{m3} that the component $B_0$ has dimension $4\lambda(\lambda-1)-3$, so we need to prove this directly.} 
\end{rem} 

\begin{rem}\label{r52}
\emph{See \cite[Theorems 5.13, 5.17]{cf} for many examples of non-empty Brill-Noether loci. Our examples do not satisfy the hypotheses in these theorems.}
\end{rem}

\section{Genus 6}\label{sec-g6}

In genus $6$, one can prove a good deal more. In this case, we have
\[\beta(2,10,5)=-4,\ \beta(2,K_C,5)=0,\ \beta(2,10,4)=\beta(2,K_C,4)=5.\]
\begin{theorem}\label{p-g6k5} Let $C$ be a general curve of genus $6$. Then
$\widetilde{B}(2,10,k)=\emptyset$ for $k\ge6$. Moreover
$B(2,10,5)=\widetilde{B}(2,10,5)=B(2,K_C,5)$ consists of a single point $E_{2,10,5}$ and $E_{2,10,5}$ is generated. 
\end{theorem}
\begin{proof}
Let $E$ be a semistable bundle of rank $2$ and degree $10$. Since $\Cliff_2(C)=\Cliff(C)=2$ by \eqref{eqc2}, $h^0(E)\le5$. This proves the first statement. By classical Brill-Noether theory, $B(1,5,3)=\emptyset$; hence $B(2,10,5)=\widetilde{B}(2,10,5)$. Note also that
$B(2,K_C,5)$ is non-empty by \cite{bf} and is smooth of dimension $0$ by \cite{te2}. Now suppose that $E\in B(2,10,5)$. By Corollary \ref{c-p1}, there exists an exact sequence
\begin{equation}\label{eq-g61}
0\lra L\lra E\lra L''\lra0,
\end{equation}
where $d_L=4$, $h^0(L)=2$, $d_{L''}=6$, $h^0(L'')=3$ and all sections of $L''$ lift to $E$. Tensoring \eqref{eq-g61} by $L$ and taking global sections, we get
\[h^0(L''\otimes L)\ge h^0(E\otimes L)-h^0(L\otimes L).\]
Since $C$ is Petri, by Lemma \ref{l-p2}, $h^0(L\otimes L)=3$, while the sequence
\[0\lra E\otimes L^*\lra E\otimes H^0(L)\lra E\otimes L\lra0\]
gives
\[h^0(E\otimes L)\ge2h^0(E)-h^0(E\otimes L^*)\ge9,\]
since $E\otimes L^*$ is stable of rank $2$ and degree $2$, so that $h^0(E\otimes L^*)\le1$ by \cite[Theorem B]{bgn}. So
\[h^0(L''\otimes L)\ge 9-3=6.\]
Since $d_{L''\otimes L}=10$, this implies that $\det E=L''\otimes L\simeq K_C$. It follows that every $E\in B(2,10,5)$ can be expressed in the form \eqref{eq-g61} with $L''=L^*\otimes K_C$. By Lemma \ref{l-hg1}, there are five choices for $L$. However, since $C$ can be embedded in a K3 surface $S$ for which $\operatorname{Pic}S$ is generated by the class of $C$, these five line bundles all determine the same bundle $E_{2,10,5}$ (see the paragraph following the statement of \cite[Th\'eor\`eme 0.1]{v}). 

Note finally that, if $E_{2,10,5}$ is not generated, then an elementary transformation yields a bundle in $B(2,9,5)$, contradicting the fact that $\Cliff_2(C)=2$.
\end{proof}

\begin{rem}\label{r-g61}
\emph{From the definition of $\Cliff_2(C)$, it follows that any bundle computing $\Cliff_2(C)$ has either degree $8$ and $h^0=4$ or degree $10$ and $h^0=5$. Hence, by Theorem \ref{p-g6k5} and \cite[Proposition 5.10]{ln3}, the only such bundles are strictly semistable bundles of degree $8$ with $h^0=4$ and the stable bundle $E_{2,10,5}$.} 
\end{rem}

\begin{cor}\label{c-g61}
For all $\alpha>0$,
\[G(\alpha;2,10,5)=\{(E,H^0(E))|E=E_{2,10,5}\}.\]
\end{cor}

\begin{proof}
If $E=E_{2,10,5}$, then $h^0(E)=5$ by Theorem \ref{p-g6k5}. Moreover, $E$ is stable, so any line subbundle has degree $\le4$ and hence $h^0\le2$. Hence $(E,H^0(E))\in G(\alpha;2,10,5)$ for all $\alpha>0$.

Conversely, suppose $(E,V)\in G(\alpha;2,10,5)$. If $E$ is not stable, then $E$ admits a line subbundle $L$ of degree $\ge5$. Hence $d_{E/L}\le5$ and $h^0(E/L)\le2$. It follows that $\dim(V\cap H^0(L))\ge3$, contradicting the $\alpha$-stability of $(E,V)$.
\end{proof}

The following proposition gives an example of a non-empty rank-$3$ Brill-Noether locus on $C$ with negative Brill-Noether number.

\begin{prop}\label{p-r3} Let $C$ be a general curve of genus $6$. Then
$B(3,10,k)=\emptyset$ for $k\ge6$. Moreover
 $B(3,10,5)$ consists of a single point $E_{3,10,5}$ and there exists a short exact sequence
\begin{equation}\label{eq-r30}
0\lra E_{3,10,5}^*\lra\mathcal{O}_C^{\oplus5}\lra E_{2,10,5}\lra0.
\end{equation}
Furthermore, for all $\alpha>0$,
\[G(\alpha;3,10,5)=\{(F,H^0(F))|F=E_{3,10,5}\}.\]
\end{prop}
\begin{proof}
By \cite[Theorem 2.1]{m} or \cite[Proposition 3.5]{ln}, we have $\Cliff_3(C)=2$. Hence any stable bundle $F$ of rank $3$ and degree $10$ has $h^0(F)\le5$.

Now recall that $E:=E_{2,10,5}$ is generated and $h^0(E)=5$. We define a bundle $F$ of rank $3$ and degree $10$ (hence slope $\mu(F)=\frac{10}3$) by the exact sequence
\begin{equation}\label{eq-r31}
0\lra F^*\lra H^0(E)\otimes\mathcal{O}_C\lra E\lra0.
\end{equation}
Dualising, we obtain
\begin{equation*}
0\lra E^*\lra H^0(E)^*\otimes\mathcal{O}_C\lra F\lra0.
\end{equation*}
Since $h^0(E^*)=0$, it follows that $h^0(F)=5$. It remains to prove that $F$ is stable.

If $L$ is a quotient line bundle of $F$, then $L$ is generated and, since $L^*\subset F^*$ and $h^0(F^*)=0$ by \eqref{eq-r31}, $h^0(L^*)=0$. Hence $h^0(L)\ge2$ and  $d_L\ge4>\mu(F)$. Now suppose that $G$ is a stable rank-2 quotient bundle of $F$. Then $G$ is generated by the image $V$ of $H^0(E)^*$ in $H^0(G)$ and $h^0(G^*)=0$, so $\dim V\ge3$. Let $K$ be the kernel of the canonical surjection $V\otimes\mathcal{O}_C\ra G$. If $\dim V=3$, then $K$ is a line bundle with $h^0(K^*)\ge3$, so $d_G=-d_K\ge6$. On the other hand, the homomorphism $E^*\ra K$ is non-zero, since otherwise $E^*$ would map into a proper direct factor of $H^0(E)^*\otimes\mathcal{O}_C$, a contradiction. Hence $K^*\subset E$ and $-d_K\le4$ by stability of $E$. This is a contradiction. It follows that $\dim V\ge4$ and hence $d_G\ge8$ since $\Cliff(G)\ge\Cliff_2(C)=2$. Thus $F$ is stable. We define $E_{3,10,5}:=F$, so that \eqref{eq-r31} becomes \eqref{eq-r30}.

Conversely, let $F\in B(3,10,5)$. We have already observed that $h^0(F)=5$. If $F$ is not generated, then, applying an elementary transformation, there exists a semistable bundle of rank $3$ and degree $9$ with $h^0=5$; this contradicts the fact that $\Cliff_3(C)=2$. We can therefore define a bundle $G$ of rank $2$ and degree $10$ by the exact sequence
\begin{equation*}
0\lra G^*\lra H^0(F)\otimes\mathcal{O}_C\lra F\lra0.
\end{equation*}
Dualising, we have
\begin{equation}\label{eq-r33}
0\lra F^*\lra H^0(F)^*\otimes\mathcal{O}_C\lra G\lra0.
\end{equation}
Now suppose $L$ is a quotient line bundle of $G$ and let $V$ be the image of $H^0(F)^*$ in $H^0(L)$. Arguing as above, we have $\dim V\ge2$. If $\dim V\ge3$, then $d_L\ge6>\mu(G)$. If $\dim V=2$, then $d_L\ge4$. Moreover, by stability of $F$, the kernel of the surjection $V\otimes\mathcal{O}_C\ra L$ is a line bundle of degree $\ge-3$, so $d_L=-d_K\le3$. This gives a contradiction, so $G$ is stable and hence $G\simeq E_{2,10,5}$. So \eqref{eq-r33} becomes \eqref{eq-r30} and $F\simeq E_{3,10,5}$.

Finally, if $F=E_{3,10,5}$, it is clear that $(F,H^0(F))\in G(\alpha;3,10,5)$ for all $\alpha>0$. Conversely, if $(F,V)\in G(\alpha;3,10,5)$ with $F$ not stable, then $F$ has either a line subbundle $L$ of degree $\ge4$ or a rank-$2$ subbundle $G$ of degree $\ge7$. In the first case, we must have $h^0(L)\le1$, so $h^0(F/L)\ge4$, which is impossible. In the second case, $h^0(G)\le3$, so $h^0(F/G)\ge2$, which again is impossible. So $F$ is stable and hence $F\simeq E_{3,10,5}$. 
\end{proof}
\begin{rem}
\emph{The fact that $B(3,K_C,5)$ has dimension zero shows that Osserman's lower bound for the dimension of the Brill-Noether locus for bundles of determinant $K_C$ when $r=3$, $k=5$ \cite[Theorem 1.1(III)]{oss} can be sharp.}
\end{rem}

\begin{rem}\label{r-g64}
\emph{When $g=6$, there is a method of constructing bundles in $B(2,10,4)$ with determinant different from $K_C$ which differs from that that in the proof of Theorem \ref{t-g6k4} (this is similar to the construction described in more generality in \cite[Theorem 5.13]{cf}, but our examples do not satisfy the hypotheses of this theorem). Consider non-trivial extensions
\begin{equation}\label{eq-g63}
0\lra L\lra E\lra L''\lra0,
\end{equation}
where $L\in B(1,4,2)$ and $L''$ is a generated line bundle of degree $6$ with $h^0(L'')=2$. If $h^0(E)=4$, then $E$ is generated and stable. In fact, if $E$ were not stable, it would admit a line subbundle $M$ of degree $5$ with $h^0(M)=2$. But then there would exist a non-zero homomorphism $M\to L''$; since $h^0(M)=h^0(L'')$, this implies that $L''$ is not generated. To obtain a bundle $E$ with $h^0(E)=4$ in \eqref{eq-g63}, we require all sections of $L''$ to lift to sections of $E$. For this, by Lemma \ref{l-p3}, we need the multiplication map 
\begin{equation}\label{eq-g64}
H^0(L'')\otimes H^0(L^*\otimes K_C)\lra H^0(L''\otimes L^*\otimes K_C)
\end{equation}
to fail to be surjective. Calculating dimensions, the LHS of \eqref{eq-g64} has dimension $6$, while the RHS has dimension $7$, so surjectivity does indeed fail. Moreover, by the base-point free pencil trick, the kernel of \eqref{eq-g64} is $H^0(L^*\otimes L''^*\otimes K_C)$, which is zero since $L''\not\simeq L_1^*\otimes K_C$. It follows that the cokernel of \eqref{eq-g64} has dimension $1$, so, for any given $L''$, the extension  is unique up to isomorphism. By classical Brill-Noether theory, the bundles $L''$ form an irreducible variety of dimension $\beta(1,6,2)=4$. Hence all such extensions belong to a single irreducible component $B_1$ of $B(2,10,4)$. Since there are five possible choices for $L$ (see Lemma \ref{l-hg1}), we obtain a possible total of five irreducible components in this way. Similarly there are ten possibilities for the component $B_0$ in Theorem \ref{t-g6k4}. If we could prove that $B_1=B_0$ (or equivalently that the bundles $E$ in \eqref{eq-g62} belong to $B_1$), then these 15 components would all coincide.}
\end{rem}

\begin{prop}\label{c-g62}
For all $\alpha>0$,
\[G(\alpha;2,10,4)=\{(E,V)|E\in B(2,10,4), V\subset H^0(E), \dim V=4\}\]
and
\[\widetilde{G}(\alpha;2,10,4)=\{[(E,V)]|[E]\in \widetilde{B}(2,10,4), V\subset H^0(E), \dim V=4\}.\]
\end{prop}

\begin{proof}
The proof is similar to that of Corollary \ref{c-g61}, the key point being that any line bundle $L$ with $d_L\le5$ has $h^0\le2$.
\end{proof}

\section{Genus 5}\label{sec-g5}

Let $C$ be a general curve of genus $5$. Since $\Cliff_2(C)=2$, it follows that $h^0(E)\le4$ for any semistable bundle of rank $2$ and degree $\le2g-2=8$. Moreover the bundles which compute $\Cliff_2(C)$ are precisely the semistable bundles of rank $2$ and degree $8$ with $h^0=4$. Note that
\[\beta(2,8,4)=1,\ \beta(2,K_C,4)=2.\] 

\begin{prop}\label{p-g51}
Let $C$ be a general curve of genus $5$. Then $B(2,K_C,4)$ is smooth of dimension $2$ and consists of the stable bundles with $h^0(E)=4$ which can be expressed in the form
\begin{equation}\label{eq-g51}
0\to M\to E\to M^*\otimes K_C\to 0, d_M=2, h^0(M)=1,
\end{equation}
with all sections of $K_C\otimes M^*$ lifting to $E$. Moreover, $B(2,K_C,4)$ is irreducible and
\begin{equation}\label{eq-g52}
\widetilde{B}(2,8,4)=\overline{B(2,K_C,4)}\cup \{[L\oplus L']|L,L'\in B(1,4,2)\}.
\end{equation}
In particular, $B(2,8,4)=B(2,K_C,4)$.
\end{prop}
\begin{proof}
The fact that $B(2,K_C,4)$ is smooth of dimension $2$ follows from \cite{bf} and \cite{te2}. A bundle $E\in B(2,K_C,4)$ cannot contain a subpencil since $B(1,3,2)=\emptyset$; it follows from \cite[Lemma 5.6]{ln3} that $E$ can be expressed in the form \eqref{eq-g51}. Now consider the multiplication map
\[m: H^0(M^*\otimes K_C)\otimes H^0(M^*\otimes K_C)\lra H^0(M^{*2}\otimes K_C^2).\]
This factors through $S^2H^0(M^*\otimes K_C)$, so $\dim(\Ker m)\ge3$. Now  $M^*\otimes K_C$ is generated and the kernel $F$ of its evaluation map has rank $2$. If $L$ is a quotient line bundle of $F^*$, then $L$ is generated and $h^0(L^*)=0$, so $d_L \ge4$; since $d_{F^*}=6$, this proves that $F$ is stable. Moreover $\Ker m\simeq H^0(F\otimes M^*\otimes K_C)\simeq H^0(F^*)$. Since $\Cliff_2(C)=2$, it follows that $\dim(\Ker m)\le3$. Hence $\dim(\Ker m)=3$ and $\dim(\Coker m)=2$. It follows from Lemma \ref{l-p3} that the isomorphism classes of non-trivial extensions \eqref{eq-g51}, for which all sections of $M^*\otimes K_C$ lift, form a $\mathbb{P}^1$-fibration $W$ over $B(1,2,1)$. Moreover, $W$ is irreducible and the open subset for which $E$ is stable maps surjectively to $B(2,K_C,4)$, which is therefore irreducible. For \eqref{eq-g52}, see \cite[Proposition 5.7]{ln3}.
\end{proof}

\begin{rem}\label{r-g51}
\emph{Mukai states that $\overline{B(2,K_C,4)}\simeq {\mathbb P}^2$ (see the table in \cite[section 4]{mu3}). Since $\beta(1,4,2)=1$, it follows from \eqref{eq-g52} that all components of $\widetilde{B}(2,8,4)$ have dimension $2$. In particular, $\widetilde{B}(2,8,4)$ has no component of dimension $\beta(2,8,4)$. This does not contradict \cite{te} since $\beta(1,4,2)=1$. Moreover $\overline{B(2,8,4)}\ne\widetilde{B}(2,8,4)$.}
\end{rem}

\begin{prop}\label{c-g51}
For all $\alpha>0$,
\[G(\alpha;2,8,4)=\{(E,H^0(E))|E\in B(2,8,4)\}\]
and
\[\widetilde{G}(\alpha;2,8,4)=\{[(E,H^0(E))]|[E]\in \widetilde{B}(2,8,4)\}.\]
\end{prop}
\begin{proof} The proof is similar to that of Corollary \ref{c-g61}.
\end{proof}

\section{Bundles computing the Clifford index}\label{sec-cliff}

By \cite[Proposition 11]{bfa}, the bundles computing $\Cliff_2(C)$ on a general curve of genus $g$ have either $h^0=4$ or degree $2g-2$ and $h^0=2+\left\lceil\frac{g-1}2\right\rceil$. This substantially improves \cite[Theorem 7.4]{ln3}. Bakker and Farkas prove further that the second possibility does not arise when $g$ is even and $g\ge10$ \cite[Theorem 4]{bfa} and conjecture that the same is true for $g$ odd, $g\ge15$. In fact, for $g$ odd, $g\ge15$, $\widetilde{B}(2,K_C,\frac{g+3}2)=\emptyset$ and all $E\in B(2,2g-2,\frac{g+3}2)$ can be expressed in the form \eqref{eq-p1} with $L\not\simeq L'$ \cite[Remark 13]{bfa}, but it is not known whether any such exist. 

For $g\le5$, there are no semistable bundles of degree $\le2g-2$ with $h^0>4$, while, for $g=6$, we have already answered the question in Remark \ref{r-g61}. It remains to consider $g=7,8,9,11,13$.

\begin{ex}\label{ex1}
\emph{Let $C$ be a general curve of genus $7$. Then $\Cliff_2(C)=3$ and we have
\[\beta(2,12,5)=0,\ \beta(2,K_C,5)=3.\]
 It is stated but not formally proved in \cite{bf} that $B(2,K_C,5)$ is non-empty. This is proved in \cite{te1} and \cite[Proposition 7.7]{ln3}. The more precise statement that $B(2,K_C,5)$ is a Fano $3$-fold of Picard number $1$ and genus $7$ is \cite[Theorem 8.1]{mu3}(see also \cite[Theorem 4.13]{mu1}); this holds for all non-tetragonal curves of genus $7$. By classical Brill-Noether theory, $B(1,6,3)=\emptyset$, so $\widetilde{B}(2,12,5)=B(2,12,5)$. Moreover, by Corollary \ref{c-p1}, any bundle $E\in B(2,12,5)\setminus B(2,K_C,5)$ can be expressed in the form \eqref{eq-p1} with $L,L'\in B(1,5,2)$, $L\not\simeq L'$. It is easy to see that any bundle given by a non-trivial extension \eqref{eq-p1}, for which all sections of $L'^*\otimes K_C$ lift, is stable, but it is not known whether such extensions exist.}
\end{ex}

\begin{ex}\label{ex2}
\emph{Let $C$ be a general curve of genus $8$. Then $\Cliff_2(C)=3$ and we have
\[\beta(2,14,6)=-7,\ \beta(2,K_C,6)=0.\]
Again, it is stated in \cite{bf} and proved in \cite[Proposition 7.2]{ln3} that $B(2,K_C,6)$ is non-empty.   By classical Brill-Noether theory, $B(1,7,3)=\emptyset$, so $\widetilde{B}(2,14,6)=B(2,14,6)$. Furthermore, $B(2,K_C,6)$ is finite by \cite{te2} and consists of a single point by \cite{v} (see also \cite[Theorem 4.14]{mu1}, where the corresponding stable bundle is described). Finally, $B(2,14,6)=B(2,K_C,6)$ by \cite[Proposition 11]{bfa}.}
\end{ex}

\begin{ex}\label{ex3}
\emph{Let $C$ be a general curve of genus $9$. Then $\Cliff_2(C)=4$ and we have
\[\beta(2,16,6)=-3,\ \beta(2,K_C,6)=3.\]
 It is stated in \cite{bf} and proved in \cite{te1} and  \cite[Proposition 7.8]{ln3} that $B(2,K_C,6)$ is non-empty. Moreover (see \cite[Theorem 7.4(1)]{ln3}), there exist strictly semistable bundles $Q\oplus Q'$ of degree $16$ with $h^0=6$. In fact, since there are just $42$ line bundles of degree $8$ with $h^0=3$, there are $21$ points of $\widetilde{B}(2,K_C,6)$ corresponding to strictly semistable bundles. In fact, Mukai \cite{mu1} asserts that $\widetilde{B}(2,K_C,6)$ is a quartic $3$-fold in $\mathbb{P}^4$ with $21$ singular points. Finally, $\widetilde{B}(2,16,6)=\widetilde{B}(2,K_C,6)$ by \cite[Proposition 11]{bfa}.}
\end{ex}

\begin{ex}\label{ex4}
\emph{Let $C$ be a general curve of genus $11$. Then $\Cliff_2(C)=5$ and we have
\[\beta(2,20,7)=-8,\ \beta(2,K_C,7)=2.\]
 It is stated in \cite{bf} and proved in \cite[Remark 7.5]{ln3} that $\widetilde{B}(2,K_C,7)$ is non-empty. More precisely, Mukai \cite[Theorem 1]{mu2} shows that $\widetilde{B}(2,K_C,7)$ is a smooth K3 surface of genus $11$. By classical Brill-Noether theory, $B(1,10,4)=\emptyset$, so $\widetilde{B}(2,20,7)=B(2, 20,7)$. Finally, $B(2,20,7)=B(2,K_C,7)$ by \cite[Proposition 11]{bfa}.}
\end{ex}

\begin{ex}\label{ex5}
\emph{Let $C$ be a general curve of genus $13$. Then $\Cliff_2(C)=6$ and we have
\[\beta(2,24,8)=-15,\ \beta(2,K_C,8)=0.\]
By classical Brill-Noether theory, $B(1,12,4)=\emptyset$, so $\widetilde{B}(2,24,8)=B(2,24,8)$, but questions of existence are not clear. In fact, $g=13$, $k=8$ is the first case in which the existence part of the Bertram-Feinberg-Mukai conjecture
\[\beta(2,K_C,k)\ge0\stackrel{?}{\Longrightarrow} B(2,K_C,k)\ne\emptyset\]
is unresolved. The arguments of \cite{bf,lnp,mu3,te1,v,zh} all fail, although \cite[Theorem 3.5]{lnp} (see also \cite{mu3}) does reduce the problem to a purely combinatorial one. Moreover, by \cite[Proposition 4.3]{v} and Lemma \ref{l-p3}, if $E\in B(2,K_C,8)$, then $E$ does not admit a subpencil and is expressible in the form 
\[0\lra L\lra E\lra L^*\otimes K_C\lra0\]
with $d_L\ge 6$ by \cite{ms}. We must have $h^0(L)\le1$, so
\[h^0(E)\le h^0(L)+h^0(L^*\otimes K_C)\le 1+(13-d_L)\le8.\]
The only way to achieve equality is with $d_L=6$ and $h^0(L)=1$ and then all sections of $L^*\otimes K_C$ must lift to $E$. On the other hand, if $E\in B(2,24,8)\setminus B(2,K_C,8)$, then $E$ is expressible in the form \eqref{eq-p1} with $L,L'\in B(1,8,2)$, $L\not\simeq L'$. It is not clear whether there are any $L$, $L'$ for which there exist non-trivial extensions of this form for which all sections of $L'^*\otimes K_C$ lift, but, if these do exist, $E$ is necessarily stable. Note that, in this case, \cite[Proposition 11]{bfa} does not apply since the general curve of genus $13$ cannot be embedded in a K3 surface.}
\end{ex}

\end{document}